\documentclass[12pt, notitlepage]{amsart}
\usepackage{latexsym, amsfonts, amsmath, amssymb, amsthm, cite}
\usepackage{graphicx}
\usepackage{mathrsfs}
\usepackage[linktocpage=true,colorlinks,citecolor=blue,linkcolor=blue,urlcolor=blue]{hyperref}
\usepackage{amssymb}
\usepackage{cite}
\usepackage{bbm}
\usepackage{amsmath}
\usepackage{latexsym}
\usepackage{amscd}
\usepackage{tikz}
\usepackage{amsthm}
\usepackage{mathrsfs}
\usepackage{url}
\usepackage[utf8]{inputenc}
\usepackage[english]{babel}
\usepackage{amsfonts}
\usepackage{mathtools}
\usepackage{comment}

\numberwithin{equation}{section}
\def\X{\mathbb X}

\newcommand{\F}{{\mathscr{F}}}

\def\N{\mathbb N}

\newtheorem{theorem}{Theorem}[section]
\newtheorem{lemma}[theorem]{Lemma}

\theoremstyle{remark}

\theoremstyle{definition}

\theoremstyle{remark}

\numberwithin{equation}{section}

\begin{document}
	\title[Dirichlet polynomials with multiplicative coefficients]{Extreme values of Dirichlet polynomials with multiplicative coefficients}
 	\author{Max Wenqiang Xu} \address{Department of Mathematics, Stanford University, Stanford, CA, USA}
 	\email{maxxu@stanford.edu}
 	 	\author{Daodao Yang} \address{Institute of Analysis and Number Theory \\ Graz University of Technology \\ Kopernikusgasse 24/II, 
A-8010 Graz \\ Austria}
 	\email{yang@tugraz.at \quad yangdao2@126.com}
	\begin{abstract}
 We study extreme values of Dirichlet polynomials with multiplicative coefficients, namely 
 \[D_N(t) : =  D_{f,\, N}(t)= \frac{1}{\sqrt{N}} \sum_{n\leqslant N} f(n) n^{it}, \]
 where $f$ is a completely multiplicative function with $|f(n)|=1$ for all $n\in\mathbb{N}$. We use Soundararajan's resonance method to produce large values of $\left|D_N(t)\right|$ uniformly for all such $f$. In particular, we improve a recent result of Benatar and Nishry, where they establish weaker lower bounds and only for almost all such $f$.  
	\end{abstract}
	\maketitle
\section{Introduction}
Resonance methods are very successful in giving extreme values of arithmetic functions. 
The earliest applications of  resonance methods can at least  be traced back to Voronin's work in \cite{Vor}. A more general and powerful version of the resonance method was introduced by Soundararajan in \cite{Sound} which was successful in finding extreme values of zeta and $L$ functions. There have been many further developments and applications of the method, we refer readers to \cite{A16, Yang22, SoundICM, A2019, BS17, BS18, CM21, BT19, Yangarx22, Yangdaodao22, A19large, siep23} and references therein. 

In this paper, we study the extreme values of Dirichlet polynomials with multiplicative coefficients. 
Let $\F$ be the set of completely multiplicative functions $f$ with $|f(n)|=1$ for all $n \in \N$. Let 
\[D_N(t) : =  D_{f,\, N}(t)= \frac{1}{\sqrt{N}} \sum_{n\leqslant N} f(n) n^{it}. \]

\begin{theorem}\label{thm: main}
Let $D_{f,\,N}(t)$ be defined as above. Let $\delta, \gamma \in (0, 1)$ be fixed. Let $T=N^{C(N)}$ where $C(N)$ satisfies $2/\delta \leqslant C(N)\leqslant (\log N)^{\gamma}$ . Then for sufficiently large $N$, we have
\[   \sup_{|t|\leqslant T} |D_{f,\,N}(t)| \geqslant \exp \left( \sqrt{(1-\delta)\frac{\log T}{\log  \log T}} \right)\,,   \]
uniformly for all $f \in \F.$ 
\end{theorem}

A Steinhaus random multiplicative function $\X(n)$ is a completely multiplicative function and $\X(p)$ are independent random variables taking value on the complex unit circle for all primes $p$. 
Another way to state Theorem~\ref{thm: main} is that $\F$ is the family of all possible Steinhaus random multiplicative functions. 
In particular, for all $\X$, and all large $N$,
\[  \sup_{|t|\leqslant N^{C(N)}} \left|\frac{1}{\sqrt{N}} \sum_{n\leqslant N} \X(n) n^{it}\right| \geqslant \exp \left( \sqrt{(1-\delta)\frac{\log T}{\log  \log T}} \right)\,\geqslant \exp \left( \sqrt{ \left( \frac{1-\delta }{1+\gamma} \right)\frac{C(N) \log N}{\log  \log N}} \right). \]
This improves the lower bound obtained in a recent interesting paper \cite[Theorem 1.1]{BN2022} where they have $(\log \log N)^{4}$ instead of our $\log \log N$ in the denominator. Our result is stronger also in the sense that the bound holds \textit{uniformly for all} $f$, rather than just \textit{almost all $f$} as in \cite[Theorem 1.1]{BN2022}. We remark that the multiplicativity of $\X(n)$ is crucial here. Without it, the extreme values are significantly smaller (see \cite[Section 3.3]{BN2022} for more discussions). We refer readers to recent results in extreme values of sums of random multiplicative functions to \cite{LTW13, Harperlargevalue, KSX,mastrostefano2020maximal}. 

\section{Proof of Theorem~\ref{thm: main}}
Our method was initialed by Soundararajan in \cite{Sound}. 
We first set up the framework by using the resonance method. 
 Let \[R(t):= \sum_{n\leqslant X}r_f(n)n^{it} \quad \text{where}\quad X= T^{1-\frac{2\delta}{3}},\]
 where $r_f(n)$ is a function defined on integers. 
As in \cite{Sound},  we define $\Phi:\, \mathbb R \to \mathbb R$ to be a smooth function, compactly supported in $[\frac{1}{2}, 1]$, 
with $0 \leqslant \Phi(y) \leqslant  1$ for all $y$, and $\Phi(y)=1$ for $5/8\leqslant y\leqslant 7/8$.
Let 
\[M_1(R, T): = \int_{-\infty}^{+\infty} |R(t)|^{2} \Phi(\frac{t}{T}) dt, \]
and 
\[ M_2 (R, T): =   \int_{-\infty}^{+\infty} |R(t)|^{2} \Phi(\frac{t}{T}) |D_N(t)|^{2} dt\,.\]
\
 Then
\begin{align}\label{Dn}
    \sup_{|t|\leqslant T} |D_N(t)| \geqslant \sqrt{\frac{M_2(R, T)}{M_1(R, T)}}.\end{align}

The quantity $M_1(R, T)$ has a nice expression.  
\begin{lemma}[\cite{Sound}]\label{prop: M1}
We have 
\begin{equation}\label{eqn: M1}
    M_1(R, T) = T{\hat \Phi}(0)  (1 + O(T^{-1}) ) \sum_{n\leqslant X} |r_f(n)|^{2}. 
\end{equation}
\end{lemma}

\begin{proof}The result is in \cite[Equation (2), page 471]{Sound}. The key point is that partial integration gives that for any positive integer $\nu$,
\begin{equation}\label{eqn: partial}
    {\hat \Phi}(y) \ll_{\nu} |y|^{-\nu}
\end{equation}
 which leads to the expression \eqref{eqn: M1}. 
\end{proof}
Now we focus on estimating $M_2(R, T)$. 
\begin{equation}
    M_2(R, T) = \frac{T}{N} \sum_{m, n\leqslant N} \sum_{a, b \leqslant X} f (n) \overline{ f (m)} r_f(a) \overline{ r_f(b)} \hat{\Phi} (T    \log (\frac{ma}{nb})).
\end{equation}
We split $M_2(R, T)$ into two parts by considering two cases $ma =nb$ or $ma\neq nb$.
\begin{equation}\label{eqn: M2}
    M_2(R, T) = \frac{T}{N} {\hat \Phi}(0)\sum_{\substack{m, n\leqslant N\\a, b \leqslant X\\ ma = nb}}  f (n) \overline{ f (m)} r_f(a) \overline{ r_f(b)}  +  \frac{T}{N}\sum_{\substack{m, n\leqslant N\\a, b \leqslant X\\ ma \neq nb}}f (n) \overline{ f (m)} r_f(a) \overline{ r_f(b)} \hat{\Phi} (T    \log (\frac{ma}{nb}))
\end{equation}
We first consider the off-diagonal terms, i.e., $ma\neq nb$. By our assumption, $T = N^{C(N)} \geqslant N^{\frac{2}{\delta}}$. So in this case, we have
\begin{align*}
    T\left| \log \frac{ma}{nb}  \right| \geqslant T\frac{1}{NX} = \frac{T^{\frac{2\delta}{3}}}{N} \geqslant T^{ \frac{\delta}{6}}.
\end{align*}
Apply \eqref{eqn: partial} to get that \begin{align*}
    \left|\hat{\Phi} (T    \log (\frac{ma}{nb}))\right| \ll_{\delta} T^{-10}\,. 
\end{align*}
This leads to
\begin{align*}
    &\Big|   \frac{T}{N}\sum_{\substack{m, n\leqslant N\\a, b \leqslant X\\ ma \neq nb}}f (n) \overline{ f (m)} r_f(a) \overline{ r_f(b)} \hat{\Phi} (T    \log (\frac{ma}{nb}))\Big| \\\ll_{\delta}\,& \frac{T}{N} \sum_{m, n\leqslant N}1 \Big( \sum_{a \leqslant X}\ |r_f(a)|\Big)^2 T^{-10}\\\ll_{\delta}\,& T^{-9}N  X \sum_{n \leqslant X}\ |r_f(n)|^2\\\ll_{\delta}\,& T^{-7}\sum_{n\leqslant X} |r_f(n)|^2\,,
\end{align*}
where in the second step we use Cauchy–Schwarz inequality.
This shows that the off-diagonal terms are negligible.

So in the next steps, we just need to consider the diagonal terms in \eqref{eqn: M2}, i.e
the sum involving terms $ma = nb$. We use the following trick to get rid of dependence on $f$. Set  \[r_f(n) =  \overline{f(n)} r (n),\] where $r$ is a non-negative multiplicative function to be chosen later, independent of $f$.
By noticing that $|f(n)|=1$, the sum involving diagonal terms is the same as 
\begin{equation}\label{eqn: M}
    \frac{T}{N} {\hat \Phi}(0)\sum_{\substack{m, n\leqslant N\\a, b \leqslant X\\ ma = nb}}   r(a)  r(b). 
\end{equation}
By the estimates on off-diagonal terms and  \eqref{eqn: M1}, we have 
\begin{equation}\label{eqn: M_2/M_1}
\frac{M_2(R, T)}{M_1(R, T)}=  \frac{1}{N} \sum_{\substack{m, n\leqslant N\\a, b \leqslant X\\ ma = nb}}   r(a) r(b)   \Big/\sum_{n\leqslant X} r(n)^2 + O_{\delta}\left(\frac{1}{T}\right)\,.
\end{equation}
Now our theorem would immediately follow from Hough's work \cite{Hough} on large character sums.
In particular, his work would imply
\begin{equation}\label{eqn: ratio}
   \frac{1}{N} \sum_{\substack{m, n\leqslant N\\a, b \leqslant X\\ ma = nb}}   r(a)  r(b)   \Big/\sum_{n\leqslant X} r(n)^2 \geqslant  \exp \left( (2+o(1)) \sqrt{ \frac{\log X}{\log\log X}} \right).
\end{equation}
For completeness, we present Hough's work. 
Let $g = (a, b)$, $h = (m ,n)$, $a = a'g, b = b'g$ then $(a', b' ) = 1$ and $m=hb', n=ha'$. With this parametrization, we have
\begin{equation}\label{eqn: 2.8}
\begin{split}
    \sum_{\substack{m, n\leqslant N\\a, b \leqslant X\\ ma = nb}}   r(a)r(b) &\geqslant \sum_{\substack{a', b' \leqslant X \\ (a', b') = 1}}r(a')r(b')\sum_{\substack{g \leqslant \frac{X}{\max(a', b')}\\ (g, a'b')=1 } }r^2(g) \sum_{h \leqslant \frac{N}{\max(a', b')}}  1\\
    & \geqslant \left(1+o(1)\right) N \sum_{\substack{a', b'\leqslant \min\{X, N\}\\(a',b')=1}}  \frac{r(a')r(b') a'b'}{\max(a',b')^3} \sum_{\substack{g \leqslant \frac{X}{\max(a', b')}\\ (g, a'b')=1 } }r^2(g).
   \\ 
\end{split}
\end{equation}
We replace the sum over $g$ by multiplicativity and Rankin's trick, for any $\alpha>0$,
\begin{align*}
\sum_{\substack{g \leqslant \frac{X}{\max(a', b')}\\ (g, a'b')=1 } }r^2(g) = & \sum_{\substack{(g, a'b')=1 } }r^2(g) -  \sum_{\substack{g > \frac{X}{\max(a', b')}\\ (g, a'b')=1 } }r^2(g)\\ = &\prod_{p \nmid a'b'}(1+r(p)^{2}) + O\left(\Big(\frac{X}{\max(a', b')}\Big)^{-\alpha} \prod_{p\nmid a'b'}(1+r(p)^{2}p^{\alpha}) \right).     
\end{align*}
We also have, by multiplicativity of $r(n)$, 
\[\sum_{n\leqslant X}r(n)^{2} \leqslant \prod_{p}(1+r(p)^{2}).\]
Combining the above two estimates and \eqref{eqn: 2.8}, we have that the main term in \eqref{eqn: ratio} is at least
\begin{equation}\label{eqn: mainterm} M:\,=
\sum_{\substack{a', b'\leqslant \min(X, N)\\(a', b')=1}}  \frac{r(a')r(b')a'b'}{\max(a', b')^{3}} \Big/ \prod_{p|a'b'}(1+r(p)^{2}),
\end{equation}
and the error term is (let $z:= \min(X, N)$) at most
\begin{equation}\label{eqn: error}
 E:\,=  \prod_{p}(1+r(p)^{2})^{-1} X^{-\alpha} \sum_{\substack{a', b'\leqslant z\\ (a',b')=1} } \frac{r(a')r(b')(a'b')^{1+\alpha}}{(a'b')^{\frac{3}{2}}} \sum_{(g, a'b')=1}r(g)^{2} g^{\alpha}. 
\end{equation}
Next, we make a choice of the resonator $r(n)$. Set $\lambda= \sqrt{\log X \log \log X }$. 
We choose $r(n)$ supported square-free integers and let 
\[r(p)= 
\begin{cases}
\frac{\lambda}{\sqrt{p}\log p}\,,& \lambda^{2} \leqslant p \leqslant \exp((\log \lambda)^{2})    \\
0\,, & \text{otherwise.}
\end{cases}
\]
We define a multiplicative function $t(n)$ supported on squarefree integers by setting $t(p)= \frac{r(p)}{1+r(p)^{2}}$. We have the following estimates borrowed from \cite{Hough}. 
\begin{lemma}[\cite{Hough}, Lemma 4.5]\label{lem: main} Uniformly in $z\geqslant 1$, 
\[    \sum_{\substack{m_1, m_2 \leqslant z \\ (m_1, m_2) =1 }} \frac{t(m_1) t(m_2) m_1 m_2}{\max(m_1 , m_2)^{3}} \geqslant \frac{1}{\log z} \Big(\sum_{m\leqslant z} \frac{t(m)}{\sqrt{m}}\Big)^{2}.  \]
  Assume that $z>\exp(3\lambda \log \log \lambda)$ with $\lambda = \sqrt{\log X\log \log X}.$ As $X \to +\infty$, we have
\[    \sum_{\substack{m_1, m_2 \leqslant z \\ (m_1, m_2) =1 }} \frac{t(m_1) t(m_2) m_1 m_2}{\max(m_1 , m_2)^{3}} \geqslant \exp\Big((1+o(1)) \frac{\lambda}{\log \lambda} \Big).  \]
\end{lemma}

\begin{lemma}[\cite{Hough}, Lemma 4.3]\label{lem: ratio}
Assume that $z>\exp(3\lambda \log \log \lambda)$ with $\lambda = \sqrt{\log X\log \log X}$ and $\alpha = (\log \lambda )^{-3}$. Then 
\begin{equation*}
    X^{-\alpha} \sum_{\substack{m_1, m_2\leqslant z\\ (m_1, m_2)=1}} \frac{r(m_1)r(m_2)}{(m_1m_2)^{1/2 - \alpha}} \sum_{(d, m_1m_2)=1} r(d)^{2}d^{\alpha} \Big/ \left(\sum_{m\leqslant X}\frac{t(m)}{\sqrt{m}}\right)^{2} \sum_{d} r(d)^{2} \leqslant \exp\Big(-(1+o(1)) \frac{32 \log X}{(\log \log X)^{4}}\Big).
\end{equation*}
\end{lemma}

By our assumption that $C(N)\leqslant (\log N)^{\gamma}$ and $X= T^{1-\frac{2\delta}{3}}= N^{C(N)(1-\frac{2\delta}{3})}$, we have $\log N > 3 \lambda \log \log \lambda$. And clearly, $\log X > 3\lambda \log \log \lambda$ for all large $X$. Apply Lemma~\ref{lem: main} to get that 
\[\sum_{\substack{a', b'\leqslant \min(X, N)\\(a', b')=1}}  \frac{r(a')r(b')a'b'}{\max(a', b')^{3}} \Big/ \prod_{p|a'b'}(1+r(p)^{2}) \geqslant \exp \left( (2+o(1)) \sqrt{ \frac{\log X}{\log\log X}} \right).\]
By Lemma~\ref{lem: ratio}, we find that  the ratio $E/M $ tends to zero as $X\to +\infty$, where the quantities $M,\, E$ are defined in \eqref{eqn: mainterm} and \eqref{eqn: error}.
Thus we complete the proof of \eqref{eqn: ratio}.

Combining \eqref{Dn}, \eqref{eqn: M_2/M_1}, \eqref{eqn: ratio} and recalling that $X= T^{1-\frac{2\delta}{3}}$, we are done.

\subsection*{Acknowledgement} The authors would like to thank Kannan Soundararajan for  interesting discussions. Yang thanks the hospitality of the math department at Stanford University during his visit when the project started. Xu is supported by the Cuthbert C. Hurd Graduate Fellowship in the Mathematical Sciences, Stanford. Yang is supported by the Austrian Science Fund (FWF), project W1230.

	\bibliographystyle{abbrv}
	\bibliography{resonance}{}
\end{document}